\documentclass[12pt]{amsart}
\usepackage{amsmath}
\usepackage{amsthm}
\usepackage{amssymb}
\usepackage{amscd}
\usepackage{hyperref}
\usepackage{tikz}
\usetikzlibrary{matrix,arrows}

\usepackage{comment}

\DeclareMathOperator{\lc}{H}

\newcommand{\lcm}[1]{\operatorname{H^{#1}_{\mathfrak m}}}

\newcommand{\Spec}{\operatorname{Spec}}
\newcommand{\length}{\ell}

\newcommand{\eh}{\operatorname{e}}

\newcommand{\mf}{\mathfrak}

\DeclareMathOperator{\depth}{depth}

\DeclareMathOperator{\m}{\mathfrak{m}}

\DeclareMathOperator{\Ann}{Ann}

\DeclareMathOperator{\Ass}{Ass}
\DeclareMathOperator{\rt}{rt}
\DeclareMathOperator{\reg}{reg}
\DeclareMathOperator{\rees}{\mathcal{R}}
\DeclareMathOperator{\AR}{ar}
\DeclareMathOperator{\NCM}{NCM}
\DeclareMathOperator{\Att}{Att}

\newcommand{\gr}{\operatorname{gr}}

\newtheorem{theorem}{Theorem}[section]
\newtheorem{lemma}[theorem]{Lemma}
\newtheorem{proposition}[theorem]{Proposition}

\newtheorem{corollary}[theorem]{Corollary}

\newtheorem*{statement*}{Statement}
\newtheorem*{theorem*}{Theorem}
\newtheorem*{lemma*}{Lemma}
\newtheorem*{fact*}{Fact}
\newtheorem*{mainthm*}{Main Theorem}

\theoremstyle{definition}
\newtheorem{definition}[theorem]{Definition}
\newtheorem*{definition*}{Definition}

\newtheorem*{example*}{Example}

\theoremstyle{remark}
\newtheorem{remark}[theorem]{Remark}
\newtheorem{question}[theorem]{Question}

\begin{document}

\title{Filter regular sequence under small perturbations}

\author{Linquan Ma}
\address{Department of Mathematics, Purdue University, West Lafayette, IN 47907 USA}
\email{ma326@purdue.edu}

\author{Pham Hung Quy}
\address{Department of Mathematics, FPT University, Hanoi, Vietnam}
\email{quyph@fe.edu.vn}

\author{Ilya Smirnov}
\address{Department of Mathematics, Stockholm University, S-106 91, Stockholm, Sweden}
\email{smirnov@math.su.se}

%\bibliographystyle{amsplain}
%\tableofcontents
\maketitle

\begin{abstract}
We answer affirmatively a question of Srinivas--Trivedi \cite{SrinivasTrivedi}: in a Noetherian local ring $(R,\m)$, if $f_1,\dots,f_r$ is a filter-regular sequence  and $J$ is an ideal such that $(f_1, \ldots, f_r)+J$ is $\m$-primary, then there exists $N>0$ such that for any $\varepsilon_1,\dots,\varepsilon_r \in \m^N$, we have an equality of Hilbert functions: $H(J, R/(f_1,\dots,f_r))(n)=H(J, R/(f_1+\varepsilon_1,\dots, f_r+\varepsilon_r))(n)$ for all $n\geq 0$. We also prove that the dimension of the non Cohen--Macaulay locus does not increase under small perturbations, generalizing another result of \cite{SrinivasTrivedi}.
\end{abstract}

\section{Introduction}

Many fundamental questions in singularity theory arise from studying deformations. One particular way of deforming a singularity is by changing the defining equations by adding terms of high order. This problem often arises while working with analytic singularities.
The first instance is the problem of finite determinacy which asks whether for a singularity defined analytically, e.g., as a quotient of a (convergent) power series ring, can be transformed into an equivalent algebraic singularity by truncating the defining equations.
More generally, such truncation arise from Artin's approximation that gives a way to descend finite structures, such as modules and finite complexes,
over the completion of a ring to finite structures over its henselization, which is a direct limit of essentially finite extensions with many nice properties.

This problem was first considered by Samuel in 1956 \cite{Samuel}, who showed 
for hypersurfaces $f \in S = k[[x_1, \ldots, x_d]]$ with isolated singularities
that for large $N$ if $\varepsilon \in \m^N$ then there is an automorphism of $S$ that maps $f \mapsto f + \varepsilon$. Samuel's result was extended by Hironaka in 1965 \cite{Hironaka}, who showed that if $S/I$ is an equidimensional reduced isolated singularity, and the ideal $I'$ obtained by changing the generators of $I$ by elements of sufficiently larger order is such that $S/I'$ is still reduced, equidimensional, and same height as $I$, then there is an automorphism of $S$ that maps $I \mapsto I'$. Cutkosky and Srinivasan further extended Samuel and Hironaka's result, we refer to \cite{CutkoskySrinivasan93,CutkoskySrinivasan97} for more details. %result to include all complete intersection prime ideals and Hironaka’s result to all reduced equidimensional ideals (by considering perturbations in the sufficiently large power of the Jacobian ideal). %See also \cite{GreuelPham} for some more recent development  showed that this is, in fact, equivalent to existence of the isomorphism $S/I \cong S/I'$. %showed that the isomorphism $I \mapsto I'$ exists for all perturbations of $I$ if and only if $I$ is a complete intersection isolated singularity.

On the other hand, by results in \cite{CutkoskySrinivasan97,GreuelPham}, in order to get an isomorphism under small perturbation, it is essential to perturb by elements contained in the Jacobian ideal. So instead of requiring the deformation to give isomorphic rings, we consider a weaker question: what properties are preserved by a sufficiently fine perturbation? Since usually singularities are measured and studied via various numerical invariants, we are led to study the behavior of such invariants.

The first such study that we are aware of was performed by Eisenbud in \cite{Eisenbud} in relation with
the application of Artin's approximation by Peskine and Szpiro in \cite{PeskineSzpiro}. Eisenbud
showed how to control the homology of a complex under a perturbation and thus showed that Euler characteristic and depth can be preserved.

Perhaps the most natural direction is to study the behavior of Hilbert--Samuel function. Srinivas and Trivedi in \cite{SrinivasTrivedi} showed that the Hilbert--Samuel function of a sufficiently fine perturbation is at most the original Hilbert--Samuel function \cite[Lemma 3]{SrinivasTrivedi}. This can be viewed as saying that the singularity of a perturbation, measured by the Hilbert--Samuel function, is no worse than the original singularity. A natural question they asked is that whether the Hilbert--Samuel functions can be actually {\it preserved} under small perturbation, when the ideal is generated by a filter-regular sequence \cite[Question 1]{SrinivasTrivedi}.

The notion of a filter-regular sequence first appeared in the theory of generalized Cohen--Macaulay rings \cite{SNN}.
Regular sequences are filter-regular, and, roughly speaking, a filter-regular sequence behaves like a regular sequence on the punctured spectrum. Every system of parameters of a generalized Cohen--Macaulay ring is a filter-regular sequence, and the converse holds under mild conditions.

Our main result in this paper answers the above question of Srinivas and Trivedi in the affirmative.

\begin{mainthm*}[Theorem \ref{main}]
Let $(R, \m)$ be a Noetherian local ring. Suppose $f_1,\dots,f_r$ is a filter-regular sequence and $J$ is an ideal such that $(f_1, \ldots, f_r)+J$ is $\m$-primary. Then there exists $N>0$ such that for every $\varepsilon_1, \varepsilon_2, \ldots, \varepsilon_r \in \m^N$, the Hilbert--Samuel functions of
$R/(f_1, \ldots, f_r)$ and $R/(f_1+\varepsilon_1, \ldots, f_r+\varepsilon_r)$ with respect to $J$ are equal. In fact,
\[
\gr_{J} (R/(f_1, \ldots, f_r)) \cong \gr_{J} (R/(f_1 + \varepsilon_1, \ldots, f_r + \varepsilon_r)).
\]
\end{mainthm*}

%They also showed that if $f_1, \ldots, f_r$ are just assumed to be filter-regular then we have a weaker equality $\eh (J, R/(f_1, \ldots, f_r)) = \eh(J, R/(f_1+\varepsilon_1, \ldots, f_r+\varepsilon_r))$ and asked in \cite[Question~1]{SrinivasTrivedi} whether the equality of Hilbert--Samuel functions holds.

%To shed more light on the theorem, we note that a regular sequence is always filter-regular (see Section 2 for more detailed definitions).
This was proved in \cite[Theorem 3]{SrinivasTrivedi} under the additional hypothesis that $\lc_{\m}^i(R/(f_1, \ldots, f_r))$ has finite length for all $i<\dim(R/(f_1, \ldots, f_r))$. To the best of our knowledge, our main theorem was not known before even when $f_1,\dots,f_r$ is a regular sequence. Let us also mention that, as pointed out in \cite{SrinivasTrivedi}, requiring $I=(f_1,\dots,f_r)$ to be generated by a filter-regular sequence is the greatest generality at which the conclusion could hold.

We want to note that for finite determinacy, explicit bounds are known in terms of the Jacobian
ideal (see for example \cite{GreuelPham}). For complete intersections that do not necessarily have an isolated singularity,
Srinivas and Trivedi obtained explicit bounds in terms of multiplicity in \cite[Proposition~1]{SrinivasTrivediCI} (see also \cite{SrinivasTrivediJAG,TrivediHilbertfunctionCMregularityuniformAR,QuyTrung} for related results).

\begin{question}\label{explicit question}
Can one obtain explicit bounds on $N$ in the Main Theorem?
\end{question}

With this question in mind, we give explicit bounds on $N$ wherever possible. However,
we do not know how to get an explicit constant in the main theorem of \cite{Eisenbud} and an explicit bound for \cite[Remark 1.12]{HunekeTrivedi}, which are needed in the proof of Theorem~\ref{main}.

This paper is organized as follows: After gathering preliminaries in Section~2, we prove our main result, Theorem~\ref{main}, in Section~3.
In contrast with \cite{SrinivasTrivedi}, we use induction on the length of the filter-regular sequence, and
as a key step in the induction we establish a uniform bound on certain Artin--Rees numbers in Corollary~\ref{uniform AR}.
In Section~4, we extend another related result of Srinivas and Trivedi by showing that the dimension of the
non Cohen--Macaulay locus can be controlled under small perturbations.

\vspace{0.5em}

\noindent\textbf{Acknowledgement}: This work has been done during a visit of the second and third authors to Purdue University in June 2019. The first author is supported in part by NSF Grant DMS $\#1901672$ and by NSF Grant DMS $\#1836867/1600198$ when preparing this paper. The second author is supported by a fund of Vietnam National Foundation for Science and Technology Development (NAFOSTED) under grant number 101.04-2020.10. The third author's
stay at the Purdue University was supported by Stiftelsen G S Magnusons fond of Kungliga Vetenskapsakademien. The authors thank the referee for her/his suggestions that led to improvement of this paper.

\section{Preliminaries}
Throughout this paper, all rings are commutative, Noetherian, with multiplicative identity, and all $R$-modules are finitely generated. We begin by collecting the definition and some basic properties of filter-regular sequences that were introduced in \cite{SNN}.

\begin{definition}
A sequence of elements $f_1,\dots,f_r$ in $R$ is called an {\it improper regular sequence} if for every $0 \leq i < r$, $f_{i+1}$ is a nonzerodivisor on $R/(f_1,\dots,f_i)$. If, in addition, $(f_1,\dots,f_r)R\neq R$, then it is called a {\it regular sequence}.
\end{definition}

\begin{definition}
Let $(R, \m)$ be a local ring. An element $f$ is called {\it filter-regular} if $f \notin P$ for $P \in \Ass (R) \setminus \{\m\}$.
A sequence of elements $f_1, \ldots, f_r$ is filter-regular if for every $0 \leq i < r$ the image of
$f_{i + 1}$ in $R/(f_1, \ldots, f_{i})$ is a filter-regular element.
\end{definition}

It should be noted that $f$ being filter-regular on $(R,\m)$ means $\mathrm{Supp}(0:f) \subseteq V(\frak m)$, or equivalently, $0:f \subseteq 0: \frak m^n$ for $n \gg 0$.

\begin{lemma}\label{filter props}
Let $f_1,\dots,f_r$ be a filter-regular sequence in a local ring $(R,\m)$. We have
\begin{enumerate}
\item The Koszul homology modules $\lc_i (f_1, f_2, \ldots, f_r; R)$ has finite length for all $i\neq 0$.
\item The image of $f_1$ in $R/(f_2, \ldots, f_r)$ is a filter-regular element.
\end{enumerate}
\end{lemma}
\begin{proof}
The first assertion is well-known, see \cite[Lemma 1.17]{SchenzelSixLectures}. To see the second assertion, we observe that $f_1,\dots,f_r$ is a filter-regular sequence if and only if the image of $f_1,\dots,f_r$ is an improper regular sequence on $R_P$ for every $P\neq\m$. Now it is enough to prove that if $f_1,\dots,f_r$ is an improper regular sequence on $R_P$, then $f_1$ is always a nonzerodivisor on $R_P/(f_2,\dots,f_r)$: if one of the $f_i$ is a unit on $R_P$ then this is clear, otherwise $f_1,\dots,f_r$ is a regular sequence in $R_P$ and thus $f_1$ is a nonzerodivisor on $R_P/(f_2,\dots,f_r)$.
\end{proof}

\begin{remark}
Clearly, a regular sequence is a filter-regular sequence. Moreover, a filter-regular sequence is an improper regular sequence on the punctured spectrum.
However, many properties of regular sequence do not apply to filter-regular sequence.
For example, the length of a filter-regular sequence is not generally bounded and, by an easy prime avoidance argument,
every $\m$-primary ideal can be generated by a filter-regular sequence. We also caution the reader that, unlike regular sequence, filter-regular sequence in a local ring does not permute in general: consider $R=k[[x,y,z]]/(xy, xz)$, then $x+y, z$ is a filter-regular sequence, but $z$ is not a filter-regular element.
\end{remark}

We next recall the (strong) Artin--Rees number. %and some results from \cite{SrinivasTrivedi} on the behavior of associated graded ring under small perturbations.

\begin{definition}
Let $R$ be a Noetherian ring, $J$ be an ideal, and $N \subset M$ be $R$-modules.
The Artin--Rees number, $\AR (J, N \subset M)$, is the least integer $s$ such that for all $n \geq s$
\[
J^nM \cap N = J^{n-s} (J^sM \cap N).
\]
If $I \subset R$ is an ideal, we simplify our notation and write $\AR_J (I) = \AR (J, I \subset R)$.
\end{definition}

%\subsection{Relation type and Castelnuovo-Mumford regularity}

We also recall that for an ideal $J\subseteq R$, the Rees algebra $\rees_J(R)$ is the subalgebra $R[Jt]$ of $R[t]$ and the associated graded ring $\gr_J (R)$ can be defined as the quotient $\rees_J (R)/J\rees_J (R)$. Both $\rees_J(R)$ and $\gr_J (R)$ are $\mathbb{N}$-graded rings generated by degree one forms over their degree zero subring.

\begin{definition}
Let $G$ be a Noetherian $\mathbb{N}$-graded ring generated over $G_0$ by forms of degree $1$.
The {\it relation type} of $G$, $\rt (G)$, is the maximum degree of a minimal generator of the ideal defining $G$ as
as a quotient of a polynomial ring over $G_0$. The {\it Castelnuovo--Mumford regularity} of $G$, $\reg(G)$, is defined to be $\max\{a_i(G)+i| i\geq 0\}$, where $a_i(G)=\max\{n|\lc_{G_+}^i(G)_n\neq 0\}$.\footnote{Here $\lc_{G_+}^i(G)$ denotes the $i$-th local cohomology module of $G$ supported at the irrelevant ideal $G_+=\oplus_{i>0}G_i$. It follows from the \v{C}ech complex characterization of local cohomology (via a homogeneous set of generators of $G_+$) that each $\lc_{G_+}^i(G)$ is $\mathbb{Z}$-graded.}
\end{definition}

Planas-Vilanova \cite{PlanasVilanova,PlanasVilanova2} pioneered an approach to uniform Artin--Rees property via relation type.
The next result follows from \cite[Theorem~2]{PlanasVilanova2}.

\begin{theorem}\label{thm relation type}
Let $I, J$ be ideals of $R$. Then $\AR_J (I) \leq \rt(\rees_J(R/I))$.
\end{theorem}

We also recall \cite[Lemma 4.8]{Ooishi}, which was also rediscovered in \cite[Proposition 4.1]{JohnsonUlrich} and \cite[Corollary~3.3]{Trung}.

\begin{theorem}\label{thm regularity}
Let $J$ be an ideal of $R$. Then we have
$$\rt (\rees_J(R)) \leq \reg (\rees_J(R)) + 1 = \reg (\gr_{J} (R)) + 1. $$
\end{theorem}

Combining the above two theorems shows that the Artin--Rees number of $I\subseteq R$ can be bounded using invariants of the associated graded ring of $R/I$. This fact will be crucial in the proof of our main result.

\section{The main result}

In this section we prove our main result that answers Srinivas--Trivedi's question \cite[Question 1]{SrinivasTrivedi}. We will need the following result on the preservation filter-regular sequence under small perturbation \cite[Remark 1.12]{HunekeTrivedi}.
\begin{lemma}\label{prevervation of filter regular}
Let $(R, \m)$ be a local ring, and $f_1,\dots,f_r$ a filter-regular sequence. Then there exists $N>0$ such that for every $\varepsilon_1, \varepsilon_2, \ldots, \varepsilon_r \in \m^N$, $f_1 + \varepsilon_1, \ldots, f_r + \varepsilon_r$ is a filter-regular sequence.
\end{lemma}

We will also need the following fundamental result of Srinivas and Trivedi \cite[Lemma~3]{SrinivasTrivedi}. We present a proof here for the sake of completeness, and to emphasize that the bound is {\it explicit}.

\begin{lemma}\label{lemma ST surjection}
Let $(R, \m)$ be a local ring, $f_1,\dots,f_r \in R$, and $J\subseteq R$ be an ideal.
Let $k:= \AR_J  ((f_1, \ldots, f_r))$.
Then for every $\varepsilon_1, \varepsilon_2, \ldots, \varepsilon_r \in J^{k+1}$
we have a surjection
\[
\gr_{J} (R/(f_1, \ldots, f_r))  \twoheadrightarrow \gr_{J} (R/(f_1 + \varepsilon_1, \ldots, f_r + \varepsilon_r)).
\]
\end{lemma}
\begin{proof}
For any ideal $I\subseteq R$, the $n$-th graded piece of $\gr_J(R/I)$ can be identified with
\[
\frac{I + J^n}{I + J^{n+1}} \cong \frac{J^n}{J^{n+1} + J^n\cap I}.
\]
Now for $n\leq k+1$, it is clear that $(f_1, \ldots, f_r) + J^{n}=(f_1 + \varepsilon_1, \ldots, f_r + \varepsilon_r) + J^{n}$ by our choice of $\varepsilon_1,\dots, \varepsilon_r$. Therefore in order to define the desired surjection,
it is enough to show that
$J^{n+1} + (f_1, \ldots, f_r) \cap J^n \subseteq J^{n+1} + (f_1 + \varepsilon_1, \ldots, f_r + \varepsilon_r) \cap J^n$
for all $n \geq k+1$.

By our choice of $k$, we know that
$J^n \cap (f_1, \ldots, f_r) = J^{n-k} ((f_1, \ldots, f_r) \cap J^{k}).$
Therefore,
\[
J^{n+1} + (f_1, \ldots, f_r) \cap J^n = J^{n-k} (J^{k+1} + (f_1, \ldots, f_r)\cap J^k)
= J^{n-k} \left (J^k \cap ((f_1, \ldots, f_r) + J^{k+1})\right ).
\]
It follows that
\begin{align*}
J^{n+1} + (f_1, \ldots, f_r) \cap J^n &= J^{n-k} \left (J^k \cap ((f_1, \ldots, f_r) + J^{k+1})\right)
\\&=J^{n-k}\left(J^k \cap ((f_1+\varepsilon_1, \ldots, f_r+\varepsilon_r) + J^{k+1})\right)
\\&=J^{n-k}\left(J^{k+1} + J^k\cap (f_1+\varepsilon_1, \ldots, f_r+\varepsilon_r)\right)
\\&\subseteq J^{n+1} + J^n\cap (f_1 + \varepsilon_1, \ldots, f_r + \varepsilon_r). \qedhere
\end{align*}
\end{proof}

As observed in \cite[Remark after Corollary 2]{SrinivasTrivedi}, it easily follows from this lemma that the
stability of the Hilbert--Samuel function under perturbation is equivalent to an isomorphism of the associated graded rings
\[\gr_{J} (R/(f_1, \ldots, f_r))  \cong \gr_{J} (R/(f_1 + \varepsilon_1, \ldots, f_r + \varepsilon_r)).\]

We may now give an answer for Question \ref{explicit question}  in the case of one element, with an explicit bound on $N$.

\begin{theorem}\label{case one element}
Let $(R, \frak m)$ be a local ring. Let $f$ be a filter-regular element and $h$ a positive integer such that $\frak m^h (0:f) = 0$. Let $J$ be an ideal such that $(f) + J$ is $\frak m$-primary, and $t$ be a positive integer such that $\frak m^t \subseteq (f) + J$. Let $k = \mathrm{ar}_J(f)$ and $N = \max \{t(k+1), h\}$. Then for every $\varepsilon \in \frak m^N$ we have
\[
\gr_{J} (R/(f)) \cong \gr_{J} (R/(f + \varepsilon)).
\]
\end{theorem}
\begin{proof}
First we note that by the definition of $N$ we have $(f) + J = (f+ \varepsilon) + J $ for every $\varepsilon \in \frak m^N$. We can replace $J$ by $(f)+J$ without affecting the corresponding associated graded rings and $\mathrm{ar}_J(f)$ will not increase. Therefore, we will assume that $J$ is a fixed $\m$-primary ideal of $R$ that contains $f$. Second, by \cite[Remark after Corollary 2]{SrinivasTrivedi}, it is enough to prove that $\length (R/(f, J^n)) = \length (R/(f + \varepsilon, J^n)$ for $n\gg0$.
By Lemma~\ref{lemma ST surjection}, it is enough to show that for any $\varepsilon \in \m^N$, we have an inequality $\length (R/(f, J^n)) \leq \length (R/(f + \varepsilon, J^n))$ for $n\gg0$.

Let $n \geq k$. By definition, $J^n \cap (f) = J^{n-k}(J^k \cap (f))$, so
$f (J^n : f) = fJ^{n-k}(J^k : f).$
Therefore $(J^n : f) \subseteq J^{n-k}(J^k:f) + (0:f)$. Since the opposite inclusion is clear, we have
$$(J^n : f) = J^{n-k}(J^k:f) + (0:f) \subseteq J^{n-k} + (0:f)$$ for all $n \ge k$. Hence, because $\varepsilon \in \m^N$, $\varepsilon\in J^k$ and $\varepsilon\cdot (0:f)=0$, we further obtain that $$(J^n : f) \subseteq J^{n - k} + (0 : f) \subseteq (J^n : \varepsilon).$$
It follows  that $(J^n : f)  = (J^n : (f, \varepsilon))\subseteq (J^n : (f + \varepsilon))$ for all $n \ge k$. Therefore
$$\length ((J^n:f)/J^n) \le \length ((J^n:(f+\varepsilon))/J^n).$$
On the other hand,
the exact sequence
\[
0 \to \frac{(J^n:f)}{J^n} \to R/J^n \xrightarrow{f} R/J^n \to R/(f,J^n) \to 0
\]
shows that $\length (R/(f, J^n)) = \length ((J^n:f)/J^n)$. Similarly, $\length (R/(f+\varepsilon, J^n)) = \length ((J^n:(f+\varepsilon))/J^n)$. Therefore
$$ \length (R/(f, J^n)) \le \length (R/(f+\varepsilon, J^n)) $$
for all $n \ge k$ and for all $\varepsilon \in \frak m^N$. The proof is complete.
\end{proof}

\begin{lemma} \label{control colon}
Let $(R, \m)$ be a local ring, and $f_1,\dots,f_r$ a filter-regular sequence. Let $h = \length(\lc_1(f_1, \ldots, f_r; R))$. Then there exists a positive integer $N$ such that for every $\varepsilon_1, \ldots, \varepsilon_r \in \m^N$, we have
$$\frak m^h \big( (f_1 + \varepsilon_1, \ldots, \widehat{f_i + \varepsilon_i}, \ldots, f_r  + \varepsilon_r) : (f_i + \varepsilon_i)\big) \subseteq (f_1 + \varepsilon_1, \ldots, \widehat{f_i + \varepsilon_i}, \ldots, f_r  + \varepsilon_r)$$
for all $1 \le i \le r$.
\end{lemma}
\begin{proof} Note that $h$ is well-defined by Lemma \ref{filter props}. By the main theorem of \cite{Eisenbud}
there exists $N$ such that for all $\varepsilon_1, \ldots, \varepsilon_r \in \m^{N}$,
$\gr (\lc_1 (f_1 +\varepsilon_1, \ldots, f_r + \varepsilon_r;R))$
is a subquotient of $\gr (\lc_1 (f_1, \ldots, f_r;R))$, where $\gr$ denotes the associated graded piece of certain filtration on the Koszul homology (see \cite{Eisenbud} for details). Therefore $\length(\lc_1(f_1 +\varepsilon_1, \ldots, f_r + \varepsilon_r;R)) \le h$ for all $\varepsilon_1, \ldots, \varepsilon_r \in \m^N$. On the other hand, for every $1\le i \le r$ we have exact sequence of Koszul homology
\begin{eqnarray*}
\lc_1 (f_1 +\varepsilon_1, \ldots, f_r + \varepsilon_r; R) &\longrightarrow& \lc_0(f_1 +\varepsilon_1, \ldots, \widehat{f_i + \varepsilon_i}, \ldots, f_r + \varepsilon_r; R)\\
&\overset{f_i + \varepsilon_i}{\longrightarrow}& \lc_0(f_1 +\varepsilon_1, \ldots, \widehat{f_i + \varepsilon_i} ,\ldots, f_r + \varepsilon_r; R).
\end{eqnarray*}
Therefore
$$
\frac{(f_1 + \varepsilon_1, \ldots, \widehat{f_i + \varepsilon_i}, \ldots, f_r  + \varepsilon_r) : (f_i + \varepsilon_i)}{(f_1 + \varepsilon_1, \ldots, \widehat{f_i + \varepsilon_i}, \ldots, f_r  + \varepsilon_r)}
$$
is a homomorphic image of $\lc_1 (f_1 + \varepsilon_1, \ldots, f_r + \varepsilon_r; R)$. Hence
\[
\length \left(\frac{(f_1 + \varepsilon_1, \ldots, \widehat{f_i + \varepsilon_i}, \ldots, f_r  + \varepsilon_r) : (f_i + \varepsilon_i)}{(f_1 + \varepsilon_1, \ldots, \widehat{f_i + \varepsilon_i}, \ldots, f_r  + \varepsilon_r)}\right) \le h.
\]
So the claim follows.
\end{proof}

Now we state and prove our main theorem.

\begin{theorem}\label{main}
Let $(R, \m)$ be a local ring. Suppose $f_1,\dots,f_r$ is a filter-regular sequence and $J$ is an ideal such that $(f_1, \ldots, f_r)+J$ is $\m$-primary. Then there exists $N>0$ such that for every $\varepsilon_1, \ldots, \varepsilon_r \in \m^N$, we have
\[
\gr_{J} (R/(f_1, \ldots, f_r)) \cong \gr_{J} (R/(f_1 + \varepsilon_1, \ldots, f_r + \varepsilon_r)).
\]
%In particular, we have an equality of Hilbert functions
%\[\length (R/(f_1, \ldots, f_r, J^n)) = \length (R/(f_1 + \varepsilon_1, \ldots, f_r + \varepsilon_r, J^n))\]
%for all $n\geq 0$.
\end{theorem}
\begin{proof}
Note that we can choose $N$ such that $(f_1, \ldots, f_r)+J = (f_1 + \varepsilon_1, \ldots, f_r+ \varepsilon_r)+J$ for all $\varepsilon_1, \ldots, \varepsilon_r \in \frak m^N$. Therefore we can replace $J$ by $(f_1, \ldots, f_r)+J$ without affecting the corresponding associated graded rings. Hence, we assume that $J$ is a fixed $\m$-primary ideal of $R$ that contains $f_1, \ldots, f_r$.

%So we assume $\m^t\subseteq J$ for some $t$ and
We use induction on $r$. The case $r=1$ was proved in Theorem \ref{case one element}  and we now assume $r \geq 2$. Applying the induction hypothesis to $R/(f_1)$, there exists $K$ such that
$
\gr_{J} (R/(f_1, \ldots, f_r)) \cong \gr_{J} (R/(f_1, f_2 + \varepsilon_2, \ldots, f_r + \varepsilon_r))
$
for all $\varepsilon_2,\dots,\varepsilon_r\in \m^{K}$. We next enlarge $K$, if necessary, to assume that $K$ also satisfies both Lemma \ref{prevervation of filter regular} and Lemma \ref{control colon}.

We set $h = \length(\lc_1(f_1, \ldots, f_r; R))$ and
\begin{eqnarray*}
k = \reg\left(\gr_J(R/(f_1,\dots,f_r))\right)+ 2 & = & \reg\big(\gr_{J} (R/(f_1, f_2 + \varepsilon_2, \ldots, f_r + \varepsilon_r))\big)+ 2\\
&\geq& \AR_J \big( (f_1)(R/{(f_2 + \varepsilon_2, \ldots, f_r + \varepsilon_r)})\big)+1.
\end{eqnarray*}
by Theorem \ref{thm relation type} and Theorem \ref{thm regularity}. Applying Lemma \ref{control colon} with $\varepsilon_1 = 0$, we have
\[\m^{h}\frac{(f_2 + \varepsilon_2, \ldots, f_r + \varepsilon_r):f_1}{(f_2 + \varepsilon_2, \ldots, f_r + \varepsilon_r)}=0.\]
 By Lemma~\ref{filter props}, $f_1$ is a filter-regular element on $R/(f_2 + \varepsilon_2, \ldots, f_r + \varepsilon_r)$.
If $\m^t \subseteq J$, by applying Theorem \ref{case one element} to $R/(f_2 + \varepsilon_2, \ldots, f_r + \varepsilon_r)$, we know that
\[
\gr_{J} (R/(f_1, f_2 + \varepsilon_2, \ldots, f_r + \varepsilon_r)) \cong
\gr_{J} (R/(f_1 + \varepsilon_1, f_2 + \varepsilon_2, \ldots, f_r + \varepsilon_r))
\]
provided that $\varepsilon_1 \in \m^{\max \{tk, h\}}$.
Therefore for all $\varepsilon_1,\dots,\varepsilon_r\in \m^{\max \{K, tk, h\}}$
\[
\gr_{J} (R/(f_1, \ldots, f_r)) \cong \gr_{J} (R/(f_1, f_2 + \varepsilon_2, \ldots, f_r + \varepsilon_r))
\cong
\gr_{J} (R/(f_1 + \varepsilon_1, f_2 + \varepsilon_2, \ldots, f_r + \varepsilon_r)).
\]
The proof is complete.
\end{proof}

As an application we obtain an upper bound for Artin--Rees number under small perturbations. We suspect that the Artin--Rees number is preserved under small perturbations but we do not have a proof at this time.

\begin{corollary}\label{uniform AR}
Let $(R, \m)$ be a local ring. Suppose $f_1,\dots,f_r$ is a filter-regular sequence and $J$ is an ideal such that $(f_1, \ldots, f_r)+J$ is $\m$-primary.  Then there exist $N>0$ such that for every $\varepsilon_1, \varepsilon_2, \ldots, \varepsilon_r \in \m^N$, we have \[\AR_J  ((f_1+\varepsilon_1, f_2 + \varepsilon_2, \ldots, f_r + \varepsilon_r)) \leq \reg\left(\gr_J(R/(f_1,\dots,f_r))\right)+1.\]
\end{corollary}
\begin{proof}
By Theorem \ref{main}, there exist $N>0$ such that for every $\varepsilon_1, \varepsilon_2, \ldots, \varepsilon_r \in \m^N$, $$\gr_J(R/(f_1,\dots,f_r))\cong \gr_J(R/(f_1+\varepsilon_1, f_2 + \varepsilon_2, \ldots, f_r + \varepsilon_r)).$$ In particular they have the same regularity and the result follows from Theorem~\ref{thm relation type} and Theorem \ref{thm regularity}.
\end{proof}

\section{Non Cohen--Macaulay locus}

In this section we prove that the dimension of the non Cohen--Macaulay locus does not increase under small perturbations. This partially extends \cite[Lemma 8]{SrinivasTrivedi}, which implies that the property of being Cohen--Macaulay on the punctured spectrum is preserved.
We begin by recalling some definitions and lemmas.

\begin{definition}
Let $(R, \mf m)$ be a local ring and $M$ be a finitely generated $R$-module of dimension $d$.
If $N$ is an $R$-module, then its annihilator $\Ann N = \{x \in R \mid xN = 0\}$ is an ideal of $R$.
Then we may define the ideal
$$\mathfrak{a}(M) = \Ann \lc_{\m}^{0}(M) \cdot \Ann \lc_{\m}^1(M) \cdots \Ann \lc_{\m}^{d-1}(M).$$
\end{definition}

We recall the following well-known Faltings' annihilator theorem, see \cite[9.6.6]{BrodmannSharp}.
\begin{lemma}\label{NCM locus}
Let $(R, \mf m)$ be a local ring that is a homomorphic image of a Cohen--Macaulay ring.
Then
$$V(\mathfrak{a}(R)) = \{P \in \mathrm{Spec}(R) \mid \mathrm{depth}R_P + \dim R/P < \dim R\}.$$
In particular, if we let $\NCM(R)$ to denote the non Cohen--Macaulay locus of $R$,
i.e., the subset of $\Spec R$ such that $R_P$ is not Cohen--Macaulay,
then $\NCM (R) \subseteq V(\mathfrak{a}(R)) $, and the equality holds iff $R$ is equidimensional.
\end{lemma}

The next lemma should also be well-known. But we include a proof as we cannot find a precise reference.

\begin{lemma}\label{CM deforms}
Let $(R, \mf m)$ be a local ring of dimension at least two and $f$ be a filter-regular element.
If $R/fR$ is Cohen--Macaulay, then $R$ is Cohen--Macaulay.
\end{lemma}
\begin{proof}
We consider the short exact sequence $$0 \to R/(0:f) \xrightarrow{\cdot f} R \to R/fR \to 0.$$
Since $f$ is filter-regular, $0:f$ has finite length. It follows that
$\lc^i_{\m}  (R/(0:f)) = \lc^i_{\m} (R)$ for $i > 0$.
Since $\dim R\geq 2$, $\dim (R/fR)\geq 1$. Thus from the long exact sequence of local cohomology we know that
$$\lc^0_{\m}(R/(0:f))\cong \lc^0_{\m}(R), \text{ and } \lc^i_{\m} (R) \overset{\cdot f}\hookrightarrow  \lc^i_{\m} (R) \text{ for $1\leq i<\dim R$}.$$
Therefore $\lc^i_{\m}(R)=0$ for all $i<\dim R$ and so $R$ is Cohen--Macaulay.
\end{proof}

We next prove a crucial proposition.

\begin{proposition}
\label{NCM locus one element case}
Let $(R, \mf m)$ be a local ring of dimension at least two that is a homomorphic image of a Cohen--Macaulay ring
and $f$ be a filter-regular element.
Then $\sqrt{\mathfrak{a}(R/fR)} = \sqrt{\mathfrak{a}(R) + (f)}$.
\end{proposition}
\begin{proof}
We first prove $\mathfrak{a}(R) \subseteq \sqrt{\mathfrak{a}(R/fR)}$. We consider the short exact sequence $$0 \to R/(0:f) \to R \to R/fR \to 0.$$
Since $0:f$ is $\mf m$-primary, it follows that
$\lc^i_{\m}  (R/(0:f)) = \lc^i_{\m} (R)$ for $i > 0$.
Thus the induced long exact sequence on local cohomology becomes
\begin{align*}
\lc^0_{\m} (R) \to \lc^0_{\m} (R/fR) \to \lc^1_{\mf m} (R)
\to \cdots  \to \lc^i_{\m} (R) \to \lc^i_{\m} (R/fR) \to \lc^{i + 1}_{\m} (R) \to \cdots.
\end{align*}
It follows from the above sequence that $\mathfrak{a}(R)^2 \subseteq  \mathfrak{a}(R/fR)$. Since $(f) \subseteq \sqrt{\mathfrak{a}(R/fR)}$, it follows that $\sqrt{\mathfrak{a}(R) + (f)}\subseteq \sqrt{\mathfrak{a}(R/fR)}$.\footnote{It should be noted that this inclusion is true when $f$ is a parameter element \cite[Remark 2.2, Lemma 3.7]{CuongQuy}}

Next we show that ${\mathfrak{a}(R/fR)} \subseteq \sqrt{\mathfrak{a}(R) + (f)}$. Let $P$ be a prime ideal such that $\mathfrak{a}(R) + (f) \subseteq P$.
If $\mathfrak{a}(R/fR) \not\subseteq P$, then by Lemma \ref{NCM locus} applied to $R/fR$, $R_P/fR_P$ is Cohen--Macaulay and $\depth  (R/fR)_P + \dim R/P = \dim R - 1$.  If $P$ has height one, then $P\neq \m$ since $\dim R\geq 2$. Thus $f$ is a nonzerodivisor on $R_P$ so $R_P$ is Cohen--Macaulay. If $P$ has height at least two, then by Lemma~\ref{CM deforms} we still have $R_P$ is Cohen--Macaulay. Hence in both cases we know that
$\depth R_P + \dim R/P = \dim R$. Therefore by Lemma~\ref{NCM locus} applied to $R$, $\mathfrak{a}(R) \nsubseteq P$, which is a contradiction.
\end{proof}

We next recall the notion of strictly filter-regular sequence \cite{NML}.

\begin{definition}
Let $(R, \m)$ be a local ring of dimension $d$. The local cohomology modules are Artinian and, therefore, have
finite secondary decompositions (see \cite[7.2]{BrodmannSharp} for details) and a finite set of {\it attached primes} $\Att(\lc_{\m}^i(R))$.
An element $f$ is called {\it strictly filter-regular} if $f \notin P$ for $P \in \cup_{i=0}^d\Att(\lc_{\m}^i(R)) \setminus \{\m\}$.
A sequence of elements $f_1, \ldots, f_r$ is called strictly filter-regular if for every $0 \leq i < r$ the image of
$f_{i + 1}$ in $R/(f_1, \ldots, f_{i})$ is a strictly filter-regular element.
\end{definition}

\begin{remark}
\label{remark on strictly filter}
By \cite[11.3.9]{BrodmannSharp}, $\Ass(M)\subseteq \cup_{i=0}^{\dim M}\Att(\lc_{\m}^i(M))$ for any finitely generated $R$-module $M$. Therefore a strictly filter regular sequence is a filter regular sequence. Moreover, if $f$ is strictly filter-regular, then $f$ is not in any minimal prime of $\mathfrak{a}(R)$ except possibly the maximal ideal.
\end{remark}

Finally we prove the main result of this section. This is proved in \cite[Lemma 8]{SrinivasTrivedi} when $R/(f_1, \ldots, f_r)$ is generalized Cohen--Macaulay, and our result removes this assumption.
\begin{theorem}
Let $(R, \mf m)$ be an equidimensional local ring that is a homomorphic image of a Cohen--Macaulay ring
and $f_1, \ldots, f_r$ be a filter-regular sequence. Then there exists $N>0$ such that for every $\varepsilon_1, \ldots, \varepsilon_r \in \m^N$, we have
\[
\dim \NCM(R/(f_1, \ldots, f_r))
\geq \dim \NCM (R/(f_1 + \varepsilon_1, \ldots, f_r + \varepsilon_r)).
\]
Moreover, the equality holds if $f_1, \ldots, f_r$ is a strictly filter-regular sequence.
\end{theorem}
\begin{proof}
By \cite[Lemma~1.11 and Remark~1.12]{HunekeTrivedi}, we can find $N$ such that
for all $\varepsilon_1, \ldots, \varepsilon_r \in \mf m^N$ we have
\begin{enumerate}
  \item $\dim R/(\mathfrak{a}(R), f_1, \ldots, f_r) \geq \dim R/(\mathfrak{a}(R), f_1 + \varepsilon_1, \ldots, f_r + \varepsilon_r).$
  \item $f_1 + \varepsilon_1, \ldots, f_r + \varepsilon_r$ is still a filter-regular sequence on $R$.
\end{enumerate}
First, if $r=\dim R$, then $f_1,\dots, f_r$ and $f_1+\varepsilon_1,\dots, f_r+\varepsilon_r$ are both system of parameters of $R$ and hence
$\dim \NCM(R/(f_1, \ldots, f_r))
=\dim \NCM (R/(f_1 + \varepsilon_1, \ldots, f_r + \varepsilon_r))=0$. If $r<\dim R$, then we can repeatedly apply Lemma~\ref{NCM locus} and Proposition \ref{NCM locus one element case} to see that
$$\NCM(R/(f_1, \ldots, f_r))=V(\mathfrak{a}(R/(f_1,\dots,f_r)))=V(\mathfrak{a}(R)+(f_1,\dots,f_r)),$$
and, similarly,
$$\NCM(R/(f_1+\varepsilon_1, \ldots, f_r+\varepsilon_r))=V(\mathfrak{a}(R/(f_1+\varepsilon_1,\dots,f_r+\varepsilon_r)))=V(\mathfrak{a}(R)+(f_1+\varepsilon_1,\dots,f_r+\varepsilon_r)).$$
Thus by (1) we know that \[
\dim \NCM(R/(f_1, \ldots, f_r))
\geq \dim \NCM (R/(f_1 + \varepsilon_1, \ldots, f_r + \varepsilon_r)).
\]

Finally, if $f_1, \ldots, f_r$ is a strictly filter-regular sequence, then by Proposition \ref{NCM locus one element case} and Remark \ref{remark on strictly filter} we know that
\[
\dim \left(R/(\mathfrak{a}(R), f_1, \ldots, f_r)\right) = \max\{0, \dim (R/\mathfrak{a}(R)) - r \} \leq \dim \left( R/(\mathfrak{a}(R),f_1+\varepsilon_1,\dots,f_r+\varepsilon_r)\right),
\]
so we have equality.
\end{proof}

\bibliographystyle{plain}
\bibliography{refs}

\end{document}